\documentclass[UTF8,12,a4paper]{amsart}

\usepackage{amsmath,amsthm,amsfonts,amssymb,amscd}
\usepackage{latexsym}
\usepackage{setspace}   %设置行距的宏包
\usepackage{fancyhdr}
\usepackage{tikz-cd}    %画交换图的宏包 
\usepackage{enumerate}
\usepackage[all]{xy}
\usepackage[mathscr]{euscript}
\usepackage{color}
\usepackage{tikz}
\usepackage{tcolorbox}
\usepackage{enumerate}

\usepackage{geometry}  %设置页边距的宏包
\usepackage{graphicx}

\usepackage{cite}
\usepackage[colorlinks]{hyperref}
\definecolor{darkgreen}{rgb}{0,0.35,0}
\newtcolorbox{mybox}[2][]{width=10cm,colback = red!5!white, colframe = green!75!black, fonttitle = \bfseries,colbacktitle = red!55!yellow, enhanced,attach boxed title to top left={yshift=-2mm},	title=#2,#1}

\theoremstyle{plain} 
\newtheorem{thm}{{Theorem}\hspace{0.05pt}}[section]
\newtheorem{prop}[thm]{{Proposition}\hspace{0.05pt}}
\newtheorem{cor}[thm]{Corollary\hspace{0.05pt}}
\newtheorem{lem}[thm]{{Lemma}\hspace{0.05pt}}

\theoremstyle{definition}
\newtheorem{dfn}[thm]{{Definition}\hspace{0.05pt}}
\newtheorem{cons}[thm]{{Construction}\hspace{0.05pt}}
\newtheorem{dfn/prop}[thm]{{Definition/Proposition}\hspace{0.05pt}}
\newtheorem{rem}[thm]{{Remark}\hspace{0.05pt}}

\newtheorem*{ack}{{Acknowledgments}\hspace{0.05pt}}

\newcommand{\B}{\mathscr{B}}
\newcommand{\C}{\mathscr{C}}
\newcommand{\D}{\mathscr{D}}

\newcommand{\An}{\mathrm{An}}

\newcommand{\Fun}{\mathrm{Fun}}

\newcommand{\Shv}{\mathrm{Shv}}

\newcommand{\RTop}{\mathrm{RTop}}

\newcommand{\LTop}{\mathrm{LTop}}

\newcommand{\X}{\mathscr{X}}
\newcommand{\Y}{\mathscr{Y}}
\newcommand{\LC}{\mathrm{LC}}
\newcommand{\Cons}{\mathrm{Cons}}

\newcommand{\Cat}{\mathrm{Cat}}

\begin{document}
	%\tiny
	%\scriptsize
	%\footnotesize
	\small 
	%\normalsize
	
	\title{An internal description of constructible objects in an $\infty$-topos}
	\date{\today}
	
	\author{Li He}
	\address{Graduate School of Mathematics, Nagoya University, 464-8602}
	\email{m21015y@math.nagoya-u.ac.jp}

	\maketitle
	
	\begin{abstract}
		We give an internal description of constructible objects in an $\infty$-topos. More precisely, 
		$P$-consctructible objects are locally constant objects internal to $\Fun(P,\An)$,
		for any noetherian poset $P$.
	\end{abstract}
	\tableofcontents
	\section{Introduction}
	The study of constructible sheaves has a very long history.
	By definition, constructible sheaves are built from
	locally constant sheaves. In this paper, we will provide
	a new point of view that constructible sheaves coincide with locally constant sheaves in the internal world.

	The $\infty$-category of constructible sheaves can help us to understand the exodromy phenomenon.
	In \cite[Appendix A]{HA},
	Lurie studied constructible sheaves and exit paths
	under the conical assumption.
	Later, in \cite[Section 3]{CJ23}, Clausen-Jansen
	used the atomic generation assumption
	on the $\infty$-category of constructible sheaves to study the exodromy,
	making the conical assumption go away.
	And Haine-Porta-Teyssier further studied
	the exodromy beyond conicality in more general setup
	in \cite{HPT24}.
	Now, in this paper,
	we will provide an internal description of the $\infty$-category
	of constructible sheaves,
	which should be helpful to understand the exodromy phenomenon.

	We let $\X$ be an $\infty$-topos.
	Recall from \cite[Appendix A.1]{HA} that
	an object $X\in \X$ is \textit{constant}, if
	$X$ lies in the essential image of the constant functor
	$\pi^*:\An\to \X$.
	An object $X\in \X$ is \textit{locally constant},
	if there exists an effective epimorphism
	$\sqcup_i U_i \twoheadrightarrow 1_\X$,
	such that 
	$(X\times U_i \to U_i)\in \X_{/U_i}$ is constant, for each $i$.
	Let $\LC(\X)\subset \X$ be the full
	subcategory of $\X$ spanned by the locally constant objects
	in $\X$.
	
	We let $\X$ be an $\infty$-topos and $P$ a poset.
	Haine-Porta-Teyssier defines the concept of
	$P$-\textit{stratification} of the $\infty$-topos $\X$
	in \cite[2.1.4]{HPT24}:
	A $P$-\textit{stratification} of $\X$
	is a geometric morphism
	$s_*:\X\to \Fun(P,\An)$.
	
	For each $p\in P$,
	Haine-Porta-Teyssier constructs the $p$-th stratum
	$\X_p$ of $s_*:\X\to \Fun(P,\An)$ via the cartesian square
	\begin{center}
		\begin{tikzcd}
			\X_p \ar[r,"i_{p*}"]\ar[d] & \X \ar[d]\\
			\An\simeq \Fun(\{p\},\An) \ar[r,"i_{p*}"] 
			& \Fun(P,\An)
		\end{tikzcd}
	\end{center}
	in $\RTop$.
	An object $X\in \X$ is $P$-\textit{constructible},
	if for each $p\in P$,
	$i_p^*(X)\in \X_p$ is locally constant.
	Let $\Cons_P(\X)\subset \X$ be the full
	subcategory of $\X$ spanned by the $P$-constructible objects
	in $\X$.
	Formally speaking, the $\infty$-category
	$\Cons_P(\X)$ of $P$-constructible objects in $\X$
	is defined by the cartesian square
	\begin{center}
		\begin{tikzcd}
			\Cons_P(\X)\ar[r]\ar[d,hook]
			& \prod_{p\in P} \LC(\X_p)\ar[d,hook]\\
			\X \ar[r] &
			\prod_{p\in P} \X_p
		\end{tikzcd}
	\end{center}
	in $\Cat_\infty$.
	
	In this paper, we define the concept of local constancy
	internal to an $\infty$-topos $\B$:
	\begin{dfn}
		We fix a geometric morphism $\pi^*:\B \to \X$ of $\infty$-topoi. An object $X\in \X$ is $\B$-\textit{locally constant}
		or \textit{locally constant internal to} $\B$,
		if there exists an effective epimorphism
		$\sqcup_i U_i \twoheadrightarrow 1_\X$,
		such that 
		$(X\times U_i \to U_i)\in \X_{/U_i}$
		lies in the essential image of the functor
		$\B \stackrel{\pi^*}{\to} \X \to \X_{/U_i}$.
		Let $\LC_\B(\X)\subset \X$ be the full subcategory spanned by
		the $\B$-locally constant objects in $\X$.
	\end{dfn}
	It's easy to get the commutative square
	\begin{center}
		\begin{tikzcd}
			\LC_\B(\X)\ar[r]\ar[d,hook]
			& \prod_{p\in P} \LC(\X_p)\ar[d,hook]\\
			\X \ar[r] &
			\prod_{p\in P} \X_p
		\end{tikzcd}
	\end{center}
	in $\Cat_\infty$.
	And hence we can get the canonical functor
	$\LC_\B(\X)\to \Cons_P(\X)$.
	Our main result says that if
	$\B\simeq \Fun(P,\An)$, where $P$ is a noetherian poset,
	then this canonical functor is an equivalence.
	\begin{thm} \label{LC_B simeq Cons_P}
		Suppose we are given a $P$-stratification
		$s_*:\X \to \Fun(P,\An)$.
		If $P$ is a noetherian poset,
		then there is a canonical equivalence
		$$\LC_{\Fun(P,\An)}  (\X)\stackrel{\sim}{\to} \Cons_P(\X). $$
	\end{thm}
	From this theorem, we know that the $\infty$-category 
	$\Cons_P(\X)$ of
	constructible sheaves are essentially the $\infty$-category
	$\LC_{\Fun(P,\An)}(\X)$ of locally constant sheaves internal to 
	$\Fun(P,\An)$.

	\begin{ack}
		The author would like to thank Lars Hesselholt 
		for some helpful suggestions.
	\end{ack}

	\section{Local constancy internal to an $\infty$-topos}
	\begin{dfn}
		We fix a geometric morphism $\pi^*:\B \to \X$ of $\infty$-topoi. An object $X\in \X$ is $\B$-\textit{locally constant}
		or \textit{locally constant internal to} $\B$,,
		if there exists an effective epimorphism
		$\sqcup_i U_i \twoheadrightarrow 1_\X$,
		such that 
		$(X\times U_i \to U_i)\in \X_{/U_i}$
		lies in the essential image of the functor
		$\B \stackrel{\pi^*}{\to} \X \to \X_{/U_i}$.
		Let $\LC_\B(\X)\subset \X$ be the full subcategory spanned by
		the $\B$-locally constant objects in $\X$.
	\end{dfn}

	\begin{rem}
		\begin{itemize}
			\item [(1)]
			If $f^*:\Y \to \X$ is a geometric morphism
			of  $\infty$-topoi in $\LTop_{\B/}$,
			then $f^*$ sends $\B$-locally constant objects
			in $\Y$ to $\B$-locally constant objects in $\X$.
			Thus it induces a functor
			$f^*:\LC_\B(\Y)\to \LC_\B(\X)$.
			\item [(2)]  If we are given the geometric morphisms
			$\B_1 \to \B_2 \to \X$ in $\LTop$,
			then 
			$\LC_{\B_1}(\X)\subset \LC_{\B_2}(\X)\subset \X.$
			\item [(3)] We have
			$\LC_\An(\X)=\LC(\X)$
			and $\LC_\X(\X)=\X$.
		\end{itemize} 
	\end{rem}

	\begin{rem}
		Suppose that we are given a commutative square
		\begin{center}
			\begin{tikzcd}
				\B_1 \ar[r]\ar[d]
				& \X_1 \ar[d,"\pi^*_{12}"]\\
				\B_2 \ar[r] & \X_2
			\end{tikzcd}
		\end{center}
		in $\LTop$.
		In this situation, the functor
		$\pi_{12}^*:\X_1\to \X_2$
		induces a functor
		$\pi_{12}^*:\LC_{\B_1}(\X_1)\to \LC_{\B_2}(\X_2)$.
		Indeed, if $X\in \X_1$ is $\B_1$-locally constant,
		then there exists an effective epimorphism
		$\sqcup U_i \twoheadrightarrow 1_{\X_1}$,
		such that $ (X\times U_i \to U_i)\in {\X_1}_{/U_i}$ is $\B_1$-constant. Hence the commutative diagram
		\begin{center}
			\begin{tikzcd}
				\B_1 \ar[r]\ar[d]
				& \X_1 \ar[d] \ar[r]  & {\X_1}_{/U_i} \ar[d] \\
				\B_2 \ar[r] & \X_2  \ar[r]
				& {\X_2}_{/\pi_{12}^* U_i}
			\end{tikzcd}
		\end{center}
		shows that $\pi_{12}^* X \in \X_2$ is $\B_2$-locally constant.
		Thus the given commutative square factors as
		\begin{center}
			\begin{tikzcd}
				\B_1 \ar[r]\ar[d]
				& \LC_{\B_1}(\X_1)\ar[r,hook]\ar[d]
				& \X_1 \ar[d]\\
				\B_2 \ar[r]
				& \LC_{\B_2}(\X_2)\ar[r,hook]
				& \X_2.
			\end{tikzcd}
		\end{center}
	\end{rem}

	Let $\B$ be an $\infty$-topos
	and $U\in \B$ a $(-1)$-truncated object.
	Recall from \cite[B.1.6]{HPT24} that
	if $j_*:\B_{/U}\to \B$ is the open geometric morphism
	and $i_*:\B_{\setminus U}\to \B$ is the complementary closed geometric morphism, then the functors
	$i^*:\B\to \B_{\setminus U} $
	and $j^*:\B\to \B_{/U}$
	exhibit $\B$ as the recollement of $\B_{\setminus U}$
	and $\B_{/U}$ in the sense of
	\cite[A.8.1]{HA}.

	\begin{cons}
		Let $\B$ be an $\infty$-topos with a closed-open recollement
		$(\mathscr{Z},\mathscr{U}).$
		Given a geometric morphism
		$\pi_*:\X\to \B$, we construct the $\infty$-topoi
		$\X_\mathscr{Z}$ and $\X_\mathscr{U}$
		via the following diagram
		\begin{center}
			\begin{tikzcd}
				\X_\mathscr{Z}\ar[r,hook,"i_*"]\ar[d]
				& \X \ar[d,"\pi_*"] & \X_\mathscr{U}
				\ar[l,hook',"j_*",swap] \ar[d]\\
				\mathscr{Z}\ar[r,hook,"i_*"]
				& \B & \mathscr{U}\ar[l,hook',"j_*",swap],
			\end{tikzcd}
		\end{center}
		with both squares cartesian in $\RTop$.
	\end{cons}

	\begin{rem} \label{recollement}
		By \cite[B.1.8]{HPT24},
		we are able to conclude that the functors 
		$i^*:\X\to \X_\mathscr{Z}$ and $j^*:\X\to \X_\mathscr{U}$
		exhibit $\X$ as the recollement of 
		$\X_\mathscr{Z}$ and $\X_\mathscr{U}$.
		In particular, the functors $i^*$ and $j^*$ are jointly conservative, and we have
		$i^* i_*\simeq \mathrm{id}$, $j^* i_* \simeq *$,
		$j^* j_!\simeq \mathrm{id}$, $i^* j_!\simeq *$.
		Moreover, for any object $A\in \X$, we have
		the cartesian square
		\begin{center}
			\begin{tikzcd}
				A \ar[r]\ar[d]
				& i_* i^* A\ar[d]\\
				j_* j^* A \ar[r]
				& i_* i^* j_* j^* A.
			\end{tikzcd}
		\end{center}
	\end{rem}

	\begin{lem} \label{loc-const-check-locally}
		Fix a geometric morphism $\pi_*:\X\to \B$.
		If $X\in \X$ is such that
		$i^*X\in \X_{\mathscr{Z}}$ is $\mathscr{Z}$-locally constant
		and
		$j^*X\in \X_{\mathscr{U}}$ is $\mathscr{U}$-locally constant,
		then $X\in \X$ is $\mathscr{B}$-locally constant.
		In other words, the square
		\begin{center}
			\begin{tikzcd}
				\LC_\B(\X)\ar[r]\ar[d,hook]
				& \LC_\mathscr{Z}(\X_\mathscr{Z})
				\times \LC_\mathscr{U}(\X_\mathscr{U})
				\ar[d,hook]\\
				\X \ar[r,"(i^*{,} j^*)"]
				& \X_\mathscr{Z}\times \X_\mathscr{U}
			\end{tikzcd}
		\end{center}
		is cartesian in $\Cat_\infty$.
	\end{lem}
	\begin{proof}
		Since $i^*X\in \X_\mathscr{Z}$
		is $\mathscr{Z}$-locally constant,
		by definition, there exists an effective epimorphism
		$\sqcup_\alpha U_\alpha \twoheadrightarrow 
		1_{\X_\mathscr{Z}}$
		such that
		$(i^* X\times U_\alpha \to U_\alpha) \in 
		(\X_{\mathscr{Z}})_{/U_\alpha}$
		lies in the essential image
		of the functor
		$\mathscr{Z}\to \mathscr{X}_\mathscr{Z}\to 
		(\X_{\mathscr{Z}})_{/U_\alpha}$.
		Similarly, 
		since $j^*X\in \X_\mathscr{U}$
		is $\mathscr{U}$-locally constant,
		there exists an effective epimorphism
		$\sqcup_\beta V_\beta \twoheadrightarrow 
		1_{\X_\mathscr{U}}$
		such that
		$(j^* X\times V_\beta \to V_\beta) \in 
		(\X_{\mathscr{U}})_{/V_\beta}$
		lies in the essential image
		of the functor $\mathscr{U}\to \mathscr{X}_\mathscr{U}\to 
		(\X_{\mathscr{U}})_{/V_\beta}$.
		
		Consider the map
		$ \sqcup_{\alpha,\beta}
		(i_* U_\alpha \times j_! V_\beta)\to 1_\X $.
		Because by Remark \ref{recollement}, 
		$i^* i_*\simeq \mathrm{id}$, $j^* i_* \simeq *$,
		$j^* j_!\simeq \mathrm{id}$, $i^* j_!\simeq *$,
		and together with the two effective epimorphisms
		$ \sqcup_\alpha U_\alpha  \twoheadrightarrow 
		1_{\X_\mathscr{Z}}$
		and
		$\sqcup_\beta V_\beta \twoheadrightarrow 1_{\X_\mathscr{U}}$,
		we see that
		$$i^*(  \sqcup_{\alpha,\beta}
		(i_* U_\alpha \times j_! V_\beta)\to 1_\X )
		\simeq
		( \sqcup_\beta (\sqcup_\alpha U_\alpha )\twoheadrightarrow 
		1_{\X_\mathscr{Z}})$$
		and
		$$
		j^*(   \sqcup_{\alpha,\beta}
		(i_* U_\alpha \times j_! V_\beta)\to 1_\X )
		\simeq ( \sqcup_\alpha( \sqcup_\beta V_\beta )
		\twoheadrightarrow 1_{\X_\mathscr{U}} ) $$ 
		are effective epimorphisms.
		Since the functors $i^*$ and $j^*$ are jointly conservative,
		by \cite[A.4.2.1]{SAG}, 
		the map
		$ \sqcup_{\alpha,\beta}
		(i_* U_\alpha \times j_! V_\beta)\twoheadrightarrow 1_\X $
		is also an effective epimorphism.
		We let $W_{\alpha \beta}:=i_* U_\alpha \times j_! V_\beta$.
		Since $i^* W_{\alpha \beta } \simeq U_\alpha$,
		we get the commutative diagram
		\begin{center}
			\begin{tikzcd}
				\B \ar[r,"\pi^*"]\ar[d,"i^*"]
				& \X \ar[r]\ar[d,"i^*"]  
				& \X_{/W_{\alpha \beta} } 
				\ar[d,"\Tilde{i}^*"]\\
				\mathscr{Z} \ar[r,"\pi^*_\mathscr{Z}"]
				& \mathscr{X}_\mathscr{Z}\ar[r]
				& (\mathscr{X}_\mathscr{Z})_{/U_\alpha}.
			\end{tikzcd}
		\end{center}
		
		By definition, there exists some $Z_\alpha\in \mathscr{Z}$
		such that
		$ (\pi^*_\mathscr{Z}Z_\alpha \times U_\alpha\to U_\alpha)
		\simeq (i^* X \times U_\alpha \to U_\alpha )$
		in $(\mathscr{X}_\mathscr{Z})_{/U_\alpha}$.
		Similarly, there exists some $P_\beta\in \mathscr{U}$
		such that
		$ (\pi^*_\mathscr{U}P_\beta \times V_\beta\to  V_\beta)
		\simeq (i^* X\times V_\beta\to  V_\beta )$
		in $(\mathscr{X}_\mathscr{U})_{/V_\beta }$.
		We claim that 
		there is an equivalence
		$$
		( \pi^* (i_* Z_\alpha \times j_! P_\beta)\times
		W_{\alpha\beta}\to W_{\alpha\beta} )
		\simeq
		(X\times W_{\alpha\beta}\to W_{\alpha\beta})
		\in \X_{/W_{\alpha\beta}},
		$$
		which finishes the proof.
		Because
		$ i^* \pi^* \simeq \pi_{\mathscr{Z}}^* i^*$
		and
		$
		(i^*X\times \times U_\alpha
		\to  U_\alpha)\simeq
		(\pi^*_\mathscr{Z}Z_\alpha \times U_\alpha
		\to  U_\alpha)\in (\X)_{U_\alpha }$,
		we get the equivalences
		\begin{align*}
			\Tilde{i}^*( \pi^* (i_* Z_\alpha \times j_! P_\beta)\times
			W_{\alpha\beta}\to W_{\alpha\beta} )
			&\simeq ( i^*\pi^* (i_* Z_\alpha \times j_! P_\beta)\times
			U_\alpha\to U_\alpha)\\
			\simeq  ( \pi^*_\mathscr{Z}Z_\alpha \times U_\alpha
			\to  U_\alpha)
			&\simeq \Tilde{i}^*( X\times W_{\alpha\beta}
			\to W_{\alpha\beta} ) .
		\end{align*}
		Similarly, we have
		$\Tilde{j}^*( \pi^* (i_* Z_\alpha \times j_! P_\beta)\times
		W_{\alpha\beta}\to W_{\alpha\beta} )
		\simeq
		\Tilde{j}^*( X\times W_{\alpha\beta}
		\to W_{\alpha\beta} )$.
		By Remark \ref{recollement}, 
		we get an equivalence
		$$ ( \pi^* (i_* Z_\alpha \times j_! P_\beta)\times
		W_{\alpha\beta}\to W_{\alpha\beta} )
		\simeq  ( X\times W_{\alpha\beta}
		\to W_{\alpha\beta} )\in \X_{/W_{\alpha\beta}} .$$
	\end{proof}

	\section{Proof of the main result}
	In this section,
	we prove our main result.
	
	We start from an useful observation.
	Let $f^*:\X\to \Y$ be a geometric morphism
	and $U$ an object in $\X$. By \cite[6.3.5.8]{HTT},
	there is a pushout square
	\begin{center}
		\begin{tikzcd}
			\X \ar[r,"f^*"]\ar[d,"\pi^*"swap]
			& \Y \ar[d,"\varphi^*"]\\
			\X_{/U}\ar[r,"F^*"]
			& \Y_{/f^* U}
		\end{tikzcd}
	\end{center}
	in $\LTop$. We have:
	\begin{lem} \label{left adjointable etale case}
		The above square is vertically left adjointable.   
		In other words, the canonical map
		$ \varphi_! F^* \to f^* \pi_!$
		is an equivalence.
	\end{lem}
	\begin{proof}
		By the direct computation.
		For any $(V\to U)\in \X_{/U}$,
		we have
		$\varphi_! F^*(V\to U) 
		\simeq \varphi_!( f^*V\to f^*U )
		\simeq f^*V\simeq f^* \pi_!(V\to U)$.
	\end{proof}

	\begin{rem}
		Given a family of geometric morphisms
		$f_i^*:\X\to \X_i$ in $\LTop$, 
		we have the functor $F= (f_i^*)_i: \X\to \prod_i \X_i$,
		and its right adjoint
		$G$ is given by
		$$
		G:\prod_i \X_i \to \X;
		(x_i)_i \mapsto \prod_i f_{i*}(x_i).
		$$
	\end{rem}

	\begin{rem} \label{check after apply to f_i^*}
		Given a family of geometric morphisms
		$f_i^*:\X\to \X_i$ in $\LTop$,
		such that the induced functor
		$F=\{f_i^*\}_i:\X \to \lim_i \X_i$ is an equivalence.
		If $A$ and $B$ in $\X$ such that there is an equivalence
		$f_i^* A\simeq f_i^* B\in \X_i$, for each $i$,
		then there is an equivalence $A\simeq B$.
		Indeed, since the right adjoint of $F$
		is given by $G:\lim_i \X_i \to \X; \{x_i\}\mapsto \lim_i f_{i*}(x_i) $, we have
		$A\simeq \lim_i f_{i*} f_i^* A \simeq
		\lim_i f_{i*} f_i^* B\simeq B$.
	\end{rem}

	\begin{prop} \label{prod of B-loc const}
		The canonical functor
		$ \LC_{\prod_i \B_i}(\prod_i \X_i )
		\to \prod_i \LC_{\B_i}(\X_i)$
		is an equivalence.
	\end{prop}
	\begin{proof}
		Suppose $\B\simeq \prod_i \B_i$
		and $\X\simeq \prod_i \X_i$.
		From the commutative square
		\begin{center}
			\begin{tikzcd}
				\LC_\B(\X)
				\ar[r]\ar[d,hook] 
				& \prod_i \LC_{\B_i}(\X_i)\ar[d,hook]\\
				\X \ar[r,"\sim"] & \prod_i \X_i.
			\end{tikzcd}
		\end{center}
		we conclude that the functor
		$\LC_\B(\X)\hookrightarrow  \prod_i \LC_{\B_i}(\X_i)$
		is fully faithful.
		It remains to show that it is essentially surjective.
		That is, given any
		$(X_i)_i\in \prod_i \X_i$
		such that each $X_i$ is $\B_i$-locally constant,
		the corresponding object
		$X\simeq \prod_i  f_{i*}X_i \in \X $
		is $\B$-locally constant.
		
		Since $X_i \in \X_i$ is $\B_i$-locally constant,
		there exists an effective epimorphism
		$\sqcup_{\alpha_i\in J_i} U_{\alpha_i}
		\twoheadrightarrow  1_{\X_i} $,
		such that
		$(X_i \times U_{\alpha_i}\to U_{\alpha_i} )
		\simeq
		(\pi_i^* A_{\alpha_i} \times U_{\alpha_i}\to U_{\alpha_i} )
		\in (\X_i)_{/U_{\alpha_i}}
		$
		for some $A_{\alpha_i}\in \B_i$,
		where $\pi_i^*:\B_i \to \X_i$.
		
		We claim that the map
		$\sqcup_{\varphi:I \to J}
		\prod_i f_{i*} U_{\alpha_i}
		\to 1_{\X} $
		is an effective epimorphism in $\X$.
		Since the family $\{f_j^*:\X\to \X_j\}_j$
		is jointly conservative,
		by \cite[A.4.2.1]{SAG},
		it suffices to show that for each $j$,
		the map
		$f_j^*( \sqcup_{\varphi:I \to J}
		\prod_i f_{i*} U_{\alpha_i})
		\to f_j^* 1_{\X}\simeq 1_{\X_j} $
		is an effective epimorphism.
		Note that we have
		\begin{align*}
			f_j^*( \sqcup_{\varphi:I \to J}
			\prod_i f_{i*} U_{\alpha_i})
			&\simeq \sqcup_{\varphi:I\to J}
			f_j^*( \prod_i f_{i*} U_{\alpha_i} )
			\simeq
			\sqcup_{\varphi:I\to J}
			f_j^* G(\{ U_{\alpha_i}\}_i )\\
			&\simeq  \sqcup_{\varphi:I\to J}  U_{\alpha_j}
			\simeq (\sqcup_{\alpha_j \in J_j} U_{\alpha_j})
			\sqcup V.
		\end{align*}
		Since $\sqcup_{\alpha_j \in J_j} U_{\alpha_j}
		\twoheadrightarrow 1_{\X_j},
		$
		the map
		$f_j^*( \sqcup_{\varphi:I \to J}
		\prod_i f_{i*} U_{\alpha_i})
		\simeq (\sqcup_{\alpha_j \in J_j} U_{\alpha_j})
		\sqcup V \twoheadrightarrow 1_{\X_j}$ 
		is also an effective epimorphism.
		
		Given functors $\varphi_i^*: \B\to \B_i$,
		we have the object
		$\prod_i \varphi_{i*}A_{\alpha_i}\in \B$.
		We claim that
		$$(\pi^* (\prod_i \varphi_{i*}A_{\alpha_i})
		\times \prod_i f_{i*} U_{\alpha_i} \to 
		\prod_i f_{i*} U_{\alpha_i})
		\simeq (X \times \prod_i f_{i*} U_{\alpha_i} \to 
		\prod_i f_{i*} U_{\alpha_i}
		)\in \X_{ / \prod_i f_{i*} U_{\alpha_i}  },
		$$
		which means that
		$X$ is $\B$-locally constant.
		By Remark \ref{check after apply to f_i^*},
		it suffices to show that for each $j$, we have an equivalence
		\begin{align*}
			&(f_j^* \pi^* (\prod_i \varphi_{i*}A_{\alpha_i})
			\times f_j^*( \prod_i f_{i*} U_{\alpha_i}) \to 
			f_j^*( \prod_i  f_{i*} U_{\alpha_i}) )\\
			&\simeq ( f_j^* X \times
			f_j^*( \prod_i f_{i*} U_{\alpha_i} )\to 
			f_j^*( \prod_i f_{i*} U_{\alpha_i} )
			).
		\end{align*}
		From the commutative square
		\begin{center}
			\begin{tikzcd}
				\B \ar[r,"\pi^*"]\ar[d,"\varphi_j^*"] 
				&  \X \ar[r]\ar[d,"f_j^*"]
				& \X_{/\prod_i f_{i*} U_{\alpha_i}}\ar[d]\\
				\B_j \ar[r,"\pi_j^*"] &  \X_j \ar[r]
				& (\X_j)_{/U_{\alpha_j}},
			\end{tikzcd}
		\end{center}
		we have
		$$
		f_j^* \pi^* (\prod_i \varphi_{i*}A_{\alpha_i})
		\simeq
		\pi_j^* \varphi_j^* (\prod_i \varphi_{i*}A_{\alpha_i})
		\simeq \pi_j^* A_{\alpha_j}.
		$$
		Since
		$f_j^* (\prod_i f_{i*} U_{\alpha_i}) 
		\simeq  U_{\alpha_j} $, we get the equivalence
		$$
		(f_j^*\pi^* (\prod_i \varphi_{i*}A_{\alpha_i})
		\times f_j^* (\prod_i f_{i*} U_{\alpha_i}) \to 
		f_j^* (\prod_i f_{i*} U_{\alpha_i}))\\
		\simeq
		(\pi_j^* A_{\alpha_j} \times U_{\alpha_j}\to U_{\alpha_j}).
		$$
		Since
		$(\pi_j^* A_{\alpha_j} \times U_{\alpha_j}\to U_{\alpha_j})
		\simeq (X_j  \times U_{\alpha_j}\to U_{\alpha_j}) $, and
		$f_j^* X \simeq f_j^*(  \prod_i f_{i*}X_i )\simeq X_j$,
		we have
		\begin{align*}
			&(f_j^*\pi^* (\prod_i \varphi_{i*}A_{\alpha_i})
			\times f_j^* (\prod_i f_{i*} U_{\alpha_i}) \to 
			f_j^* (\prod_i f_{i*} U_{\alpha_i}))\\
			\simeq&
			(\pi_j^* A_{\alpha_j} \times U_{\alpha_j}\to U_{\alpha_j})
			\simeq  (X_j  \times U_{\alpha_j}\to U_{\alpha_j})\\
			\simeq  &
			(f_j^* X
			\times f_j^* (\prod_i f_{i*} U_{\alpha_i}) \to 
			f_j^* (\prod_i f_{i*} U_{\alpha_i})),
		\end{align*}
		which finishes the proof.

	\end{proof}

	\begin{cor} \label{reduce to connected case}
		Suppose that $\B \simeq \prod_i \B_i $ in $\LTop$,
		and that each geometric morphism
		$\varphi_i^*:\B\to \B_i$ is \'etale.
		For each $i$, the $\infty$-topos
		$\X_i$ is defined by the pushout square
		\begin{center}
			\begin{tikzcd}
				\B \ar[d,"\varphi_i^*"swap]\ar[r,"\pi^*"]
				& \X \ar[d,"f_i^*"]\\
				\B_i \ar[r,"\pi_i^*"] & \X_i
			\end{tikzcd}
		\end{center}
		in $\LTop$. If $X\in \X$ such that
		$f_i^*X\in \X_i$ is $\B_i$-locally constant,
		for each $i$,
		then $X\in \X$ is $\B$-locally constant.
		In other words, the square 
		\begin{center}
			\begin{tikzcd}
				\LC_\B(\X)
				\ar[r]\ar[d,hook]
				& \prod_i \LC_{\B_i}(\X_i)\ar[d,hook]\\
				\X \ar[r] & \prod_i \X_i
			\end{tikzcd}
		\end{center}
		is cartesian in $\Cat_\infty$.
	\end{cor}
	\begin{proof}
		Since $\varphi_i^*:\B \to \B_i$ is \'etale,
		there exists $U_i\in \B$ such that $\B_i\simeq \B_{/U_i}$.
		Thus $\X_i \simeq \X_{/ \pi^* U_i}$.
		From the pushout square
		\begin{center}
			\begin{tikzcd}
				\B \ar[r,"\pi^*"]\ar[d]
				& \X \ar[d]\\
				\prod \B_{/U_i} \simeq \B_{/ \sqcup U_i } \ar[r]
				& \X_{/\pi^*(\sqcup_i  U_i)  }
				\simeq \prod \X_{/ \pi^* U_i },
			\end{tikzcd}
		\end{center}
		and the equivalence
		$\B\stackrel{\sim}{\to}\prod \B_i \simeq \prod \B_{/U_i}$,
		we know that the canonical functor
		$\X\to \prod \X_i \simeq \prod \X_{/\pi^* U_i}$
		is an equivalence.
		By Proposition \ref{prod of B-loc const},
		the functor
		$\LC_\B(\X)\to \prod \LC_{\B_i}(\X_i) $
		is an equivalence.
		Thus the square   
		\begin{center}
			\begin{tikzcd}
				\LC_\B(\X)
				\ar[r,"\sim"]\ar[d,hook]
				& \prod_i \LC_{\B_i}(\X_i)\ar[d,hook]\\
				\X \ar[r,"\sim"] & \prod_i \X_i
			\end{tikzcd}
		\end{center}
		is cartesian in $\Cat_\infty$.
	\end{proof}

	We will use the following remark in 
	Proposition \ref{connected case}.
	\begin{rem} \label{useful remark}
		\begin{itemize}
			\item [(1)]  We let $\{f_i:\C \to \D_i\}_i$ be a family of right adjoints.
			If its limit is $f:\C\to \D\simeq \lim_i \D_i$,
			then $f$ admits a left adjoint given by
			$f^L:\D \to \C; \{d_i\}_i \mapsto
			\mathrm{colim}_i  f_i^L(d_i)$,
			where $f_i^L$ is a left adjoint to $f_i$.
			\item [(2)]
			Given a commutative square
			\begin{center}
				\begin{tikzcd}
					\X \ar[r,"f^*"]\ar[d]
					& \Y \ar[d]\\
					\X_{/U}\ar[r,"\Tilde{f}^*"]
					& \Y_{/f^* U}
				\end{tikzcd}
			\end{center}
			in $\LTop$,
			if $f^*:\X\to \Y$ admits a left adjoint
			$f_!:\Y \to \X$,
			then $\Tilde{f}^*:\X_{/U}\to \Y_{/f^* U}$
			admits a left adjoint given by
			$$ \Tilde{f}_!:
			\Y_{/f^* U}\to\X_{/U};
			(V\to f^* U)\to (f_! V\to U).
			$$
			\item [(3)] Suppose we are given a family of maps
			$\{f_i^*:\X\to \Y_i \}_i$ in $\LTop$, and $U\in \X$,
			which induces the family
			$ \{ \Tilde{f}_i^*:\X_{/U}\to (\Y_i)_{/ f_i^* U} \}_i$
			whose limit is
			$\Tilde{f}^*: \X_{/U}\to \lim_i \Y_{i / f_i^* U}$.
			If each $f_i^*$ admits a left adjoint
			$f_{i!}:\Y_i \to \X$, then by (1) and (2), 
			$\Tilde{f}^*$ admits a left adjoint given by
			$$ \Tilde{f}_!:
			\lim_i \Y_{i / f_i^* U} \to \X_{/U};
			\{V_i \to f_i^* U\}_i \mapsto
			\mathrm{colim}_i ( f_{i!} V_i \to U ). 
			$$
		\end{itemize}
	\end{rem}

	\begin{prop} \label{connected case}
		Suppose that
		$\Fun(P,\An)\simeq \B \simeq \lim_i \B_i
		\simeq \lim_i \Fun(P_{\geq i},\An)$ in $\LTop$,
		where $P$ is a noetherian poset, 
		which is connected under the Alexandrov topology.
		For each $i$, the $\infty$-topos
		$\X_i$ is defined by the pushout square
		\begin{center}
			\begin{tikzcd}
				\B \ar[d,"\varphi_i^*"swap]\ar[r,"\pi^*"]
				& \X \ar[d,"f_i^*"]\\
				\B_i \ar[r,"\pi_i^*"] & \X_i
			\end{tikzcd}
		\end{center}
		in $\LTop$.
		If $X\in \X$ such that
		$f_i^*X\in \X_i$ is $\B_i$-locally constant,
		for each $i$,
		then $X\in \X$ is $\B$-locally constant.
		In other words, the square
		\begin{center}
			\begin{tikzcd}
				\LC_\B(\X)\ar[r]\ar[d,hook]
				& \lim_i \LC_{\B_i}(\X_i)\ar[d,hook]\\
				\X \ar[r] & \lim_i \X_i
			\end{tikzcd}
		\end{center}
		is cartesian in $\Cat_\infty$.
		In addition, the canonical functor
		$\LC_\B(\X)\to \lim_i \LC_{\B_i}(\X_i)$
		is an equivalence.
	\end{prop}
	
	\begin{proof}
		Note that
		since $P_{\geq i}\subset P$ is an open subset,
		the geometric morphism
		$\varphi_i^*:\B\simeq \Fun(P,\An)
		\to \Fun(P_{\geq i},\An)\simeq \B_i$ is \'etale, for each $i$.
		By \cite[6.3.5.8]{HTT}, 
		as the pushout, $f_i^*:\X\to \X_i$ is also \'etale.
		Since $\varphi_i^*:\B \to \B_i$ is \'etale,
		we can suppose $\B_i\simeq \B_{/B_i}$
		for some $B_i\in \B$, 
		and hence
		we have $\X_i \simeq \X_{/\pi^* B_i}$.
		We have the following pushout
		\begin{center}
			\begin{tikzcd}
				\B \ar[r,"\pi^*"]\ar[d]
				& \X \ar[d]\\
				\lim \B_{/B_i} \simeq
				\B_{/ \mathrm{colim} U_i } \ar[r]
				& \X_{/\pi^*(\mathrm{colim}  B_i)  }
				\simeq \lim \X_{/ \pi^* B_i }
			\end{tikzcd}
		\end{center}
		in $\LTop$.
		Since $\B \to \lim \B_{/B_i}\simeq \B_i$
		is an equivalence,
		the canonical functor
		$$(f_i^*)_i: 
		\X\to \lim \X_{/ \pi^* B_i }\simeq \lim \X_i $$
		is also an equivalence.
		If we have already known that the square
		\begin{center}
			\begin{tikzcd}
				\LC_\B(\X)\ar[r]\ar[d,hook]
				& \lim_i \LC_{\B_i}(\X_i)\ar[d,hook]\\
				\X \ar[r] & \lim_i \X_i
			\end{tikzcd}
		\end{center}
		is cartesian in $\Cat_\infty$, 
		then the functor
		$\X\to \lim_i \X_i$
		is an equivalence
		implies that the functor
		$ \LC_\B(\X)\to \lim_i \LC_{\B_i}(\X_i)$
		is an equivalence.
		Thus it remains to show the square is cartesian in 
		$\Cat_\infty$.

		We let $X\in \X$ is such that
		$f_i^* X\in \X_i$ is $\B_i$-locally constant, for each $i$. 
		Because $f_i^*X\in \X_i$
		is $\B_i$-locally constant,
		by definition, there exists an effective epimorphism
		$\sqcup_{\alpha_i\in I_i}
		U_{\alpha_i}\twoheadrightarrow 1_{\X_i}$,
		such that
		$ (f_i^* X\times U_{\alpha_i}\to U_{\alpha_i} )
		\simeq (\pi_i^* A_{\alpha_i}\times U_{\alpha_i}
		\to U_{\alpha_i})\in (\X_i)_{/U_{\alpha_i}} $,
		for some $A_{\alpha_i}\in \B_i$.
		
		We first claim that
		the map
		$\sqcup_{i,\alpha_i\in I_i} 
		f_{i!} U_{\alpha_i}\to 1_\X $
		is an effective epimorphism.
		Since the canonical functor
		$(f_j^*)_j:\X\to \lim_j \X_j $
		is an equivalence, we know that
		the family $\{f_j^*:\X\to \X_j\}_j$ is jointly conservative,
		by \cite[A.4.2.1]{SAG},
		it suffices to show that for each $j$,
		the map
		$$ f_j^*( \sqcup_{i,\alpha_i\in I_i} 
		f_{i!} U_{\alpha_i})\simeq
		\sqcup_{i,\alpha_i\in I_i}  f_j^* f_{i!}U_{\alpha_i}
		\simeq
		\sqcup_{i\neq j,\alpha_i\in I_i}
		f_j^* f_{i!}U_{\alpha_i}
		\sqcup (\sqcup_{\alpha_j} U_{\alpha_j})
		\to f_j^* 1_\X \simeq 1_{\X_j}$$
		is an effective epimorphism. 
		Since $\sqcup_{\alpha_j} U_{\alpha_j}
		\twoheadrightarrow 1_{\X_j} $,
		and it factors as
		$$\sqcup_{\alpha_j} U_{\alpha_j}\to
		\sqcup_{i\neq j,\alpha_i\in I_i}
		f_j^* f_{i!}U_{\alpha_i}
		\sqcup (\sqcup_{\alpha_j} U_{\alpha_j})
		\to  1_{\X_j},$$
		we know that
		$
		f_j^*( \sqcup_{i,\alpha_i\in I_i} 
		f_{i!} U_{\alpha_i}) \twoheadrightarrow 1_{\X_j} $.
		
		Now, we have the commutative diagram
		\begin{center}
			\begin{tikzcd}
				\B \ar[r,"\pi^*"]\ar[d,"\varphi_i^*"swap]
				& \X \ar[r]\ar[d,"f_i^*"]
				& \X_{/W_{\alpha_i}}\ar[d]\\
				\B_i \ar[r,"\pi_i^*"]
				& \X_i \ar[r]
				& {\X_i}_{/U_{\alpha_i}},
			\end{tikzcd}
		\end{center}
		where $W_{\alpha_i}\simeq f_{i!}U_{\alpha_i}$,
		and hence $U_{\alpha_i}\simeq f_i^* f_{i!}U_{\alpha_i}
		\simeq f_i^* W_{\alpha_i} $.
		Next, we claim that
		$$
		(X\times W_{\alpha_i} \to W_{\alpha_i} )
		\simeq
		(\pi^* \varphi_{i!} A_{\alpha_i}\times W_{\alpha_i}
		\to  W_{\alpha_i}  )
		\in \X_{/W_{\alpha_i}},
		$$
		which finishes the proof.
		Since we have the equivalence
		$\X\stackrel{\sim}{\to}\lim_j \X_j$,
		by Remark \ref{check after apply to f_i^*},
		it suffices to show that for each $j$, we have
		an equivalence
		$$
		(f_j^* X\times f_j^* W_{\alpha_i} \to f_j^* W_{\alpha_i} )
		\simeq
		(f_j^* \pi^* \varphi_{i!} A_{\alpha_i}\times f_j^* W_{\alpha_i}
		\to f_j^* W_{\alpha_i}  )
		\in {\X_j}_{/f_j^* W_{\alpha_i}}.
		$$
		Case I. \\
		If $j=i$, 
		since $\varphi_{i!}$ is fully faithful,
		we have the equivalences
		$
		(f_i^* \pi^* \varphi_{i!} A_{\alpha_i}
		\times  U_{\alpha_i}\to U_{\alpha_i} )
		\simeq 
		(\pi_i^* \varphi_i^* \varphi_{i!} A_{\alpha_i} 
		\times  U_{\alpha_i}\to U_{\alpha_i} )
		\simeq 
		(\pi_i^* A_{\alpha_i} 
		\times  U_{\alpha_i}\to U_{\alpha_i} )
		\simeq 
		(f_i^* X \times  U_{\alpha_i}\to U_{\alpha_i} )\in
		(\X_i)_{/U_{\alpha_i}}.
		$\\
		Case II. \\
		If $j>i$, we have the commutative diagram
		\begin{center}
			\begin{tikzcd}
				\B \ar[r,"\varphi_i^*"]\ar[d,"\pi^*"]
				& \B_i \ar[r,"\varphi_{ij}^*"]\ar[d,"\pi_i^*"]
				& \B_j \ar[d,"\pi_j^*"]\\
				\X \ar[r,"f_i^*"]
				& \X_i \ar[r,"f_{ij}^*"]
				& \X_j.
			\end{tikzcd}
		\end{center}
		In this case, we have
		$ f_j^* W_{\alpha_i}\simeq f_{ij}^* f_i^* W_{\alpha_i} 
		\simeq f_{ij}^* U_{\alpha_i}$.
		Applying the functor
		$f_{ij}^*:(\X_i)_{/ U_{\alpha_i}} \to  
		(\X_j)_{ / f_j^* W_{\alpha_i} }$
		to the equivalence
		$$
		(f_i^* \pi^* \varphi_{i!} A_{\alpha_i}
		\times  U_{\alpha_i}\to U_{\alpha_i} )
		\simeq 
		(f_i^* X \times  U_{\alpha_i}\to U_{\alpha_i} )
		\in  (\X_i)_{/ U_{\alpha_i}}
		$$
		from Case I, we get the equivalence
		$$
		(f_j^* \pi^* \varphi_{i!} A_{\alpha_i}\times f_j^* W_{\alpha_i}
		\to f_j^* W_{\alpha_i}  )
		\simeq 
		( f_j^* X  \times f_j^* W_{\alpha_i}
		\to f_j^* W_{\alpha_i}  )\in (\X_j)_{/ f_j^* W_{\alpha_i} }.
		$$
		Case III.\\
		If $j<i$, we have the commutative diagram
		\begin{center}
			\begin{tikzcd}
				\B \ar[r,"\varphi_j^*"]\ar[d,"\pi^*"]
				& \B_j \ar[r,"\varphi_{ji}^*"]\ar[d,"\pi_j^*"]
				& \B_i \ar[d,"\pi_i^*"]\\
				\X \ar[r,"f_j^*"]
				& \X_j \ar[r,"f_{ji}^*"]
				& \X_i.
			\end{tikzcd}
		\end{center}
		Since $\varphi_i^*$ and $\varphi_j^*$ are \'etale,
		by \cite[6.3.5.9]{HTT},
		$\varphi_{ji}^*:\B_j \to \B_i$ is \'etale. 
		Thus by Lemma \ref{left adjointable etale case}, 
		we have $ \pi_j^* \varphi_{ji!}\simeq f_{ji!}\pi_i^*$.
		Now we have the equivalences
		\begin{align*}
			&(f_j^* \pi^* \varphi_{i!} A_{\alpha_i}\times f_j^* W_{\alpha_i}
			\to f_j^* W_{\alpha_i}  )
			\simeq 
			(\pi_j^* \varphi_j^*
			\varphi_{i!} A_{\alpha_i}\times f_j^* W_{\alpha_i}
			\to f_j^* W_{\alpha_i}  )\\
			\simeq  &
			(\pi_j^* \varphi_j^*
			\varphi_{j!} \varphi_{ji!} A_{\alpha_i}\times f_j^* W_{\alpha_i}
			\to f_j^* W_{\alpha_i}  )
			\simeq
			(\pi_j^*  \varphi_{ji!} A_{\alpha_i}\times f_j^* W_{\alpha_i}
			\to f_j^* W_{\alpha_i}  )\\
			\simeq &  
			( f_{ji!} \pi_i^* A_{\alpha_i}\times f_j^* W_{\alpha_i}
			\to f_j^* W_{\alpha_i}  ).
		\end{align*}
		Again, since $f_{ji}^*:\X_j \to \X_i$ is \'etale, 
		by Lemma \ref{left adjointable etale case}, 
		there is a commutative square
		\begin{center}
			\begin{tikzcd}
				\X_i \ar[r]\ar[d,"f_{ji!}"swap]
				&  {\X_i}_{/U_{\alpha_i}}\ar[d,"\Tilde{f}_{ji!}"]\\
				\X_j \ar[r]
				& {\X_j}_{ / f_j^* W_{\alpha_i} }.
			\end{tikzcd}
		\end{center}
		The object
		$( f_{ji!} \pi_i^* A_{\alpha_i}\times f_j^* W_{\alpha_i}
		\to f_j^* W_{\alpha_i}  )\in {\X_j}_{ / f_j^* W_{\alpha_i} }$
		is the image of $ \pi_i^* A_{\alpha_i}\in \X_i $
		under the functor
		$\X_i \to \X_j \to  {\X_j}_{ / f_j^* W_{\alpha_i} }$,
		thus it is also the image of
		$
		(f_i^* X  \times U_{\alpha_i}\to U_{\alpha_i})
		\simeq
		(\pi_i^* A_{\alpha_i}  \times U_{\alpha_i}\to U_{\alpha_i})
		\in {\X_i}_{/U_{\alpha_i}} $
		under the functor
		$ \Tilde{f}_{ji_!}:
		{\X_i}_{/U_{\alpha_i}} \simeq
		{\X_i}_{/ f_{ji}^* f_j^* W_{\alpha_i}}
		\to
		{\X_j}_{ / f_j^* W_{\alpha_i} }$.
		By Remark \ref{useful remark},
		the functor $\Tilde{f}_{ji_!}$ assigns
		$A\to U_{\alpha_i}\simeq f_i^* W_{\alpha_i}
		\simeq f_{ji}^* f_j^* W_{\alpha_i} $
		to $ f_{ji!} A\to  f_j^* W_{\alpha_i}$,
		we get the equivalences
		\begin{align*}
			&  \Tilde{f}_{ji_!}
			(f_i^* X  \times U_{\alpha_i}\to U_{\alpha_i})
			\simeq   \Tilde{f}_{ji_!}
			(f_i^* X  \times  U_{\alpha_i}  \to 
			f_{ji}^* f_j^* W_{\alpha_i}) \\
			\simeq   &
			(f_{ji!}( f_i^* X \times U_{\alpha_i} ) 
			\to f_j^* W_{\alpha_i} )
			\simeq  
			(f_{ji!}( f_{ji}^* f_j^* X  \times  U_{\alpha_i} ) 
			\to f_j^* W_{\alpha_i} )   \\
			\simeq   &
			( f_j^* X  \times f_{ji!} U_{\alpha_i} 
			\to f_j^* W_{\alpha_i} ) 
			\simeq   
			( f_j^* X \times f_j^* W_{\alpha_i} 
			\to f_j^* W_{\alpha_i} ) ,
		\end{align*}
		here the final equivalence
		$  f_{ji!} U_{\alpha_i}  \simeq f_j^* W_{\alpha_i} $
		is because
		$$ f_j^* W_{\alpha_i}
		\simeq f_j^* f_{i!}U_{\alpha_i}
		\simeq  f_j^* f_{j!} f_{ji!} U_{\alpha_i} 
		\simeq f_{ji!} U_{\alpha_i}  .$$
		Therefore, for $j<i$, we have
		$$
		(f_j^* X\times f_j^* W_{\alpha_i} \to f_j^* W_{\alpha_i} )
		\simeq
		(f_j^* \pi^* \varphi_{i!} A_{\alpha_i}\times f_j^* W_{\alpha_i}
		\to f_j^* W_{\alpha_i}  )
		\in {\X_j}_{/f_j^* W_{\alpha_i}}.
		$$
		Case IV.\\
		The case that there exists some $k$ such that
		$k> i$ and $k> j$.
		We need to show that
		$
		(f_j^* X\times f_j^* W_{\alpha_i} \to f_j^* W_{\alpha_i} )
		\simeq
		(f_j^* \pi^* \varphi_{i!} A_{\alpha_i}\times f_j^* W_{\alpha_i}
		\to f_j^* W_{\alpha_i}  )
		\in \X_{/f_j^* W_{\alpha_i}}.
		$
		There is a pushout square
		\begin{center}
			\begin{tikzcd}
				\B \ar[r,"\varphi_i^*"]\ar[d,"\varphi_j^*"swap]
				&  \B_i \ar[d,"\Phi_i^*"]\\
				\B_j \ar[r,"\Phi_j^*"]
				&  \lim_{k\geq i,j} \B_k
			\end{tikzcd}
		\end{center}
		in $\LTop$,
		where
		$\Phi_i^*: \B_i \to \lim_{k\geq i,j} \B_k;
		x \mapsto \{ \varphi_{ik}^*(x) \}_k$,
		which admits a left adjoint
		$ \Phi_{i!}:\lim_{k\geq i,j} \B_k\to  \B_i;
		\{x_k\}_k \mapsto \mathrm{colim}_k \varphi_{ik !}(x_k)$.
		The above square is horizontally left adjointable.
		That is, the canonical map
		$ \Phi_{j!}\Phi_i^* \to  \varphi_j^* \varphi_{i!}$
		is an equivalence. In other words,
		we have the equivalence
		$$
		\varphi_j^* \varphi_{i!}
		\simeq
		\mathrm{colim}_{k\geq i,j} \varphi_{jk!} \varphi_{ik}^*.
		$$
		Also,
		since each $f_i^*: \X\to \X_i$ is \'etale, 
		we have the pushout square
		\begin{center}
			\begin{tikzcd}
				\X \ar[r,"f_i^*"]\ar[d,"f_j^*"swap]
				&  \X_i \ar[d,"F_i^*"]\\
				\X_j \ar[r,"F_j^*"]
				&  \lim_{k\geq i,j} \X_k
			\end{tikzcd}
		\end{center}
		in $\LTop$, 
		where the functor $F_j^*$ is given by
		$F_j^*: \X_j \to  \lim_{k\geq i,j}\X_k;
		x \mapsto  \{ f_{jk}^*(x) \}_{k},$
		and its left adjoint $F_{j!}$ is given 
		$F_{j!}:\lim_{k\geq i,j}\X_k\to \X_j;
		\{x_k\}_k \mapsto \mathrm{colim}_k f_{jk!}(x_k).$
		And thus we get the equivalences
		$$f_j^* f_{i!}\simeq
		F_{j!} F_i^* \simeq
		\mathrm{colim}_{k\geq i,j}f_{jk!} f_{ik}^*.$$
		By the equivalence
		$\varphi_j^* \varphi_{i!}
		\simeq
		\mathrm{colim}_{k\geq i,j} \varphi_{jk!} \varphi_{ik}^*$,
		we get
		\begin{align*}
			f_j^* \pi^* \varphi_{i!} A_{\alpha_i}
			\simeq
			\pi_j^* \varphi_j^* \varphi_{i!} A_{\alpha_i}
			\simeq
			\pi_j^*( \mathrm{colim}_{k\geq i,j} 
			\varphi_{jk!} \varphi_{ik}^* A_{\alpha_i}  )
			\simeq
			\mathrm{colim}_{k\geq i,j} \pi_j^*
			\varphi_{jk!} \varphi_{ik}^* A_{\alpha_i}.
		\end{align*}
		Since the square
		\begin{center}
			\begin{tikzcd}
				\B_j \ar[r,"\pi_j^*"]\ar[d,"\varphi_{jk}^*"swap]
				& \X_j \ar[d,"f_{jk}^*"]\\
				\B_k \ar[r,"\pi_k^*"]
				& \X_k
			\end{tikzcd}
		\end{center}
		is vertically left adjointable, 
		we have $ \pi_j^* \varphi_{jk!}
		\simeq f_{jk!}\pi_k^* $. Thus
		$$
		f_j^* \pi^* \varphi_{i!} A_{\alpha_i}
		\simeq 
		\mathrm{colim}_{k\geq i,j}
		f_{jk!}\pi_k^* \varphi_{ik}^* A_{\alpha_i}
		\simeq 
		\mathrm{colim}_{k\geq i,j}
		f_{jk!} f_{ik}^* \pi_i^*  A_{\alpha_i}.
		$$
		Because $F_j^*: \X_j \to  \lim_{k\geq i,j} \X_k$
		is \'etale, and
		$$
		F_j^*( f_j^* W_{\alpha_i} )
		\simeq 
		\{ f_{jk}^*f_j^* f_{i!}U_{\alpha_i}  \}_k
		\simeq 
		\{ f_k^* f_{i!}U_{\alpha_i}  \}_k
		\simeq 
		\{ f_{ik}^* f_i^* f_{i!}U_{\alpha_i}  \}_k
		\simeq 
		\{ f_{ik}^* U_{\alpha_i}  \}_k,
		$$
		we get the following pushout square
		\begin{center}
			\begin{tikzcd}
				\X_j \ar[r]\ar[d,"F_j^*"]
				& (\X_j)_{/ f_j^* W_{\alpha_i}  }
				\ar[d,"\Tilde{F}_j^*"] 
				\\
				\lim_{k\geq i,j} \X_k \ar[r]
				&(  \lim_{k\geq i,j} \X_k)_{/ 
					(f_{ik}^* U_{\alpha_i} )_k }\simeq
				\lim_{k\geq i,j}  (\X_k)_{/ f_{ik}^* U_{\alpha_i} }
			\end{tikzcd}
		\end{center}
		in $\LTop$,
		which is vertically left adjointable,
		by Lemma \ref{left adjointable etale case}.
		Note that the image of
		$\{f_k^* X\}_k  \in \lim_{k\geq i,j} \X_k $
		under the functor
		$ \lim_{k\geq i,j} \X_k\to
		\lim_{k\geq i,j}  (\X_k)_{/ f_{ik}^* U_{\alpha_i} } $
		is
		$ \{ f_k^* X \times f_{ik}^* U_{\alpha_i}
		\to f_{ik}^* U_{\alpha_i} \}_k $,
		and the image of
		$\{f_{ik}^* \pi_i^* A_{\alpha_i} \}_k
		\in \lim_{k\geq i,j} \X_k $
		under the functor
		$ \lim_{k\geq i,j} \X_k\to
		\lim_{k\geq i,j}  (\X_k)_{/ f_{ik}^* U_{\alpha_i} } $
		is
		$ \{ f_{ik}^* \pi_i^* A_{\alpha_i}
		\times f_{ik}^* U_{\alpha_i}
		\to f_{ik}^* U_{\alpha_i} \}_k $.
		By assumption we have
		$$( \pi_i^* A_{\alpha_i}\times U_{\alpha_i}
		\to U_{\alpha_i} )
		\simeq   ( f_i^* X \times U_{\alpha_i}
		\to U_{\alpha_i} )\in (\X_i)_{/ U_{\alpha_i}},$$
		applying the functor
		$f_{ik}^*: (\X_i)_{/ U_{\alpha_i}}
		\to (\X_k)_{/ f_{ik}^* U_{\alpha_i}}$
		to this equivalence, 
		since
		$ f_{ik}^* U_{\alpha_i}
		\simeq  f_{ik}^* f_i^* W_{\alpha_i}
		\simeq f_k^* W_{\alpha_i}$,
		we get the equivalences
		\begin{align*}
			&   (  f_k^* X \times f_k^* W_{\alpha_i}
			\to f_k^* W_{\alpha_i} )
			\simeq
			(  f_k^* X \times f_{ik}^* U_{\alpha_i}
			\to f_{ik}^* U_{\alpha_i} )\\
			\simeq &  
			( f_{ik}^* \pi_i^* A_{\alpha_i}
			\times f_{ik}^* U_{\alpha_i}
			\to f_{ik}^* U_{\alpha_i} )
			\simeq
			( f_{ik}^* \pi_i^* A_{\alpha_i}
			\times f_k^* W_{\alpha_i} 
			\to f_k^* W_{\alpha_i}  )\in(\X_k)_{/ f_{ik}^* U_{\alpha_i}}.
		\end{align*}
		We denote the functor
		$ \Tilde{F}_{j!}:
		\lim_{k\geq i,j} (\X_k)_{/  f_{ik}^* U_{\alpha_i}}
		\to (\X_j)_{/f_j^*W_{\alpha_i} } $,
		which is the left adjoint of $\Tilde{F}_j^*$.
		Next, we compute 
		the images of
		$\{  f_k^* X \times f_k^* W_{\alpha_i}
		\to f_k^* W_{\alpha_i}  \}_k$
		and
		$
		\{ f_{ik}^* \pi_i^* A_{\alpha_i}
		\times f_k^* W_{\alpha_i} 
		\to f_k^* W_{\alpha_i}  \}_k
		$
		under the functor $ \Tilde{F}_{j!}$,
		which are the same. \\
		For 
		$\{  f_k^* X \times f_k^* W_{\alpha_i}
		\to f_k^* W_{\alpha_i}  \}_k
		\in  \lim_{k\geq i,j} (\X_k)_{/ f_k^* W_{\alpha_i} }$,
		by Remark \ref{useful remark}, 
		its image under $ \Tilde{F}_{j!}$ is
		$$
		\mathrm{colim}_{k\geq i,j}
		( f_{jk!}(f_k^* X \times f_k^* W_{\alpha_i})\to f_j^* W_{\alpha_i}).
		$$
		By the projection formula for \'etale geometric morphisms,
		we have
		$$
		f_{jk!}(f_k^* X \times f_k^* W_{\alpha_i})
		\simeq f_{jk!}(f_{jk}^* f_j^* X \times f_k^* W_{\alpha_i})
		\simeq f_j^* X \times   f_{jk!}f_k^* W_{\alpha_i}.
		$$
		Because $k>i$, we have
		$$
		f_{jk!}f_k^* W_{\alpha_i}
		\simeq f_{jk!} f_k^* f_{i!} U_{\alpha_i}
		\simeq f_{jk!} f_{ik}^* f^*_i f_{i!} U_{\alpha_i}
		\simeq  f_{jk!} f_{ik}^*  U_{\alpha_i}.
		$$
		Thus we get the equivalences
		\begin{align*}
			&\mathrm{colim}_{k\geq i,j}
			( f_{jk!}(f_k^* X \times f_k^* W_{\alpha_i})
			\to f_j^* W_{\alpha_i})
			\simeq 
			\mathrm{colim}_{k\geq i,j}
			( f_j^* X \times f_{jk!} f_{ik}^*  U_{\alpha_i}  \to f_j^* W_{\alpha_i})\\
			\simeq &
			( f_j^* X \times  
			\mathrm{colim}_{k\geq i,j} 
			f_{jk!} f_{ik}^*  U_{\alpha_i}  \to f_j^* W_{\alpha_i})
			\simeq
			( f_j^* X \times  
			f_j^* f_{i!}  U_{\alpha_i}  \to f_j^* W_{\alpha_i})\\
			\simeq &
			( f_j^* X \times  
			f_j^*W_{\alpha_i}  \to f_j^* W_{\alpha_i}).
		\end{align*}
		For
		$\{ f_{ik}^* \pi_i^* A_{\alpha_i}
		\times f_k^* W_{\alpha_i} 
		\to f_k^* W_{\alpha_i}  \}_k  \in
		\lim_{k\geq i,j} (\X_k)_{/ f_k^* W_{\alpha_i} }$,
		by Remark \ref{useful remark}, 
		its image under $ \Tilde{F}_{j!}$ is
		$$
		\mathrm{colim}_{k\geq i,j}
		(f_{jk!} (f_{ik}^* \pi_i^* A_{\alpha_i}
		\times f_k^* W_{\alpha_i} )
		\to f_j^* W_{\alpha_i}  ).
		$$
		Again, by the projection formula, we have
		$$
		f_{jk!} (f_{ik}^* \pi_i^* A_{\alpha_i}
		\times f_k^* W_{\alpha_i} )
		\simeq
		f_{jk!} (f_{ik}^* \pi_i^* A_{\alpha_i}
		\times f_{jk}^* f_j^* W_{\alpha_i} )
		\simeq
		f_{jk!} f_{ik}^* \pi_i^* A_{\alpha_i}
		\times f_{j}^*  W_{\alpha_i} .
		$$
		Thus we get
		\begin{align*}
			&    \mathrm{colim}_{k\geq i,j}
			(f_{jk!} (f_{ik}^* \pi_i^* A_{\alpha_i}
			\times f_k^* W_{\alpha_i} )
			\to f_j^* W_{\alpha_i}  )
			\simeq 
			\mathrm{colim}_{k\geq i,j}
			( f_{jk!} f_{ik}^* \pi_i^* A_{\alpha_i}
			\times f_{j}^*  W_{\alpha_i}
			\to f_j^* W_{\alpha_i}  ) \\
			\simeq &
			( ( \mathrm{colim}_{k\geq i,j}
			f_{jk!} f_{ik}^* \pi_i^* A_{\alpha_i})
			\times f_{j}^*  W_{\alpha_i}
			\to f_j^* W_{\alpha_i}  ) 
			\simeq
			(  f_j^* f_{i!} \pi_i^* A_{\alpha_i}
			\times f_{j}^*  W_{\alpha_i}
			\to f_j^* W_{\alpha_i}  ) \\
			\simeq &
			(  f_j^*\pi^* \varphi_{i!} A_{\alpha_i}
			\times f_{j}^*  W_{\alpha_i}
			\to f_j^* W_{\alpha_i}  ).
		\end{align*}
		Now,
		we can conclude that we have the equivalence
		$$
		( f_j^* \pi^* \varphi_{i!} A_{\alpha_i} 
		\times   f_j^* W_{\alpha_i}\to  f_j^* W_{\alpha_i} )
		\simeq
		(  f_j^* X 
		\times   f_j^* W_{\alpha_i}\to  f_j^* W_{\alpha_i} ).
		$$
		Case V.\\
		Finally, we consider the case that
		$j$ such that there exists no $k$ such that
		$k\geq i$ and $k\geq j$.
		Because $P$ is connected, 
		there exists some $k$ such that
		$i> k$ and $j> k$.
		Since $k<i$,
		by Case III, we have the equivalence
		$$
		(f_k^* X\times f_k^* W_{\alpha_i} \to f_k^* W_{\alpha_i} )
		\simeq
		(f_k^* \pi^* \varphi_{i!} A_{\alpha_i}\times f_k^* W_{\alpha_i}
		\to f_k^* W_{\alpha_i}  )
		\in {\X_k}_{/f_k^* W_{\alpha_i}}.
		$$
		Since $j>k$, applying the functor
		$f_{kj}^*: \X_{k}\to \X_{j}$ to this equivalence,
		we get the equivalence
		$$
		(f_j^* X\times f_j^* W_{\alpha_i} \to f_j^* W_{\alpha_i} )
		\simeq
		(f_j^* \pi^* \varphi_{i!} A_{\alpha_i}\times f_j^* W_{\alpha_i}
		\to f_j^* W_{\alpha_i}  )
		\in {\X_j}_{/f_j^* W_{\alpha_i}}.
		$$
	\end{proof}

	\begin{cor} \label{the case P_{geq p}}
		If we are given a geometric morphism
		$\B \simeq \Fun(P_{\geq p},\An) \to \X$,
		where $P$ is a noetherian poset and $p\in P$, 
		then the induced square
		\begin{center}
			\begin{tikzcd}
				\LC_\B(\X)\ar[r]\ar[d,hook] &
				\prod_{q\geq p}
				\LC(\X_q) \ar[d,hook] \\
				\X \ar[r]
				& \prod_{q\geq p} \X_q
			\end{tikzcd}
		\end{center}
		is cartesian in $\Cat_\infty$. 
	\end{cor}
	\begin{proof}
		We prove it by noetherian induction on $p\in P$.
		Suppose it is true for
		$ \Fun(P_{\geq q},\An)$, where $q>p$.
		We need to show that it is also true for
		$ \Fun(P_{\geq p},\An)$.
		Note that we have
		$$
		P_{\geq p}\setminus \{p\}=
		\bigcup_{q>p} P_{\geq q},
		$$
		where $\{p\}\subset P_{\geq p}$ is closed and
		$\bigcup_{q>p} P_{\geq q}\subset P_{\geq p}$
		is open.
		Thus by Lemma \ref{loc-const-check-locally},
		we have the cartesian square
		\begin{center}
			\begin{tikzcd}
				\LC_\B(\X)\ar[r]\ar[d,hook]
				& \LC(\X_p )
				\times \LC_\mathscr{U}(\X_\mathscr{U})
				\ar[d,hook]\\
				\X \ar[r]
				& \X_p\times \X_\mathscr{U},
			\end{tikzcd}
		\end{center}
		where $\B\simeq \Fun( P_{\geq p},\An )$
		and $\mathscr{U}\simeq 
		\Fun( \bigcup_{q>p} P_{\geq q},\An  )
		\simeq \lim_{q>p} \Fun( P_{\geq q},\An )
		\simeq  \lim_{q>p} \B_q  . $
		By Proposition \ref{connected case},
		the inclusion
		$\LC_\mathscr{U}(\X_\mathscr{U})\hookrightarrow \X_\mathscr{U}$
		is exactly the inclusion
		$$
		\lim_{q>p} 
		\LC_{\B_q}(\X_{\B_q})\hookrightarrow
		\lim_{q>p}  \X_{\B_q}.
		$$
		By the induction assumption, for each $q>p$, we have the cartesian square
		\begin{center}
			\begin{tikzcd}
				\LC_{\B_q}(\X_{\B_q})\ar[r]
				\ar[d,hook] &\prod_{r\geq q} \LC( \X_r) \ar[d,hook]  \\
				\X_{\B_q}\ar[r]
				& \prod_{r\geq q} \X_r.
			\end{tikzcd}
		\end{center}
		Taking the limit over the index set $\{q>p\}$, we get the cartesian square
		\begin{center}
			\begin{tikzcd}
				\LC_\mathscr{U}(\X_\mathscr{U}) \simeq
				\lim_{q>p} \LC_{\B_q}(\X_{\B_q})\ar[r]
				\ar[d,hook] &
				\lim_{q>p} \prod_{r\geq q} \LC( \X_r) 
				\simeq \prod_{r>p}\LC( \X_r) 
				\ar[d,hook]  \\
				\X_\mathscr{U}\simeq
				\lim_{q>p} \X_{\B_q}\ar[r]
				&\lim_{q>p} \prod_{r\geq q} \X_r
				\simeq \prod_{r>p} \X_r.
			\end{tikzcd}
		\end{center}
		Hence we get the cartesian square
		\begin{center}
			\begin{tikzcd}
				\LC_\B(\X)\ar[r]\ar[d,hook]
				& \LC(\X_p )
				\times \LC_\mathscr{U}(\X_\mathscr{U})\ar[r]
				\ar[d,hook] & \prod_{r\geq p}\LC( \X_r) 
				\ar[d,hook]  \\
				\X \ar[r]
				& \X_p\times \X_\mathscr{U} \ar[r]
				&  \prod_{r\geq p} \X_r. 
			\end{tikzcd}
		\end{center}

	\end{proof}

	\begin{thm} \label{cart for LC_B}
		If we are given a geometric morphism
		$\pi^*:\B\simeq \Fun(P,\An)\to \X$,
		where $P$ is a noetherian poset, 
		then the induced square
		\begin{center}
			\begin{tikzcd}
				\LC_\B(\X)\ar[r]\ar[d,hook]
				& \prod_{p\in P} \LC(\X_p)\ar[d,hook]\\
				\X \ar[r] &
				\prod_{p\in P} \X_p
			\end{tikzcd}
		\end{center}
		is cartesian in $\Cat_\infty$.   
	\end{thm}
	\begin{proof}
		By Corollary \ref{reduce to connected case},
		we are reduced to the case the topological space
		$P$ is connected.
		Write $P=\bigcup_{i\in P} P_{\geq i}$,
		then we have
		$\B\simeq \Fun(P,\An)\simeq \lim_i \Fun( P_{\geq i},\An )
		\simeq \lim_i \B_i$.
		Since $P_{\geq i}\subset P$ is open,
		the induced geometric morphism
		$\varphi_i^*: \B\simeq\Fun(P,\An)
		\simeq \Shv(P)
		\to \Shv( P_{\geq i} )\simeq
		\Fun(P_{\geq i},\An)\simeq \B_i $ is \'etale.
		By Corollary \ref{the case P_{geq p}},
		for $\B_i\simeq \Fun(P_{\geq i},\An)$,
		we have the cartesian square
		\begin{center}
			\begin{tikzcd}
				\LC_{\B_i}(\X_i)\ar[r]\ar[d,hook]
				& \prod_{p\geq i} \LC(\X_p)\ar[d,hook]\\
				\X_i \ar[r]
				& \prod_{p\geq i} \X_p.
			\end{tikzcd}
		\end{center}
		Taking the limit over $i\in P$,
		by Proposition \ref{connected case},
		we get the cartesian square
		\begin{center}
			\begin{tikzcd}
				\LC_\B(\X)\ar[d,hook] \ar[r,"\sim"] &
				\lim_i \LC_{\B_i}(\X_i)\ar[r]\ar[d,hook]
				&\lim_i \prod_{p\geq i} \LC(\X_p)
				\simeq \prod_{p\in P} \LC(\X_p)
				\ar[d,hook]\\
				\X \ar[r,"\sim"] &   \lim_i \X_i  \ar[r]
				& \lim_i\prod_{p\geq i} \X_p
				\simeq \prod_{p\in P} \X_p .
			\end{tikzcd}
		\end{center}  
	\end{proof}

	\begin{thm} 
		Suppose we are given a $P$-stratification
		$s_*:\X \to \Fun(P,\An)$.
		If $P$ is a noetherian poset,
		then there is a canonical equivalence
		$$\LC_{\Fun(P,\An)}  (\X)\stackrel{\sim}{\to} \Cons_P(\X). $$
	\end{thm}
	\begin{proof}
		By \cite[2.2.4]{HPT24}, the square
		\begin{center}
			\begin{tikzcd}
				\Cons_P(\X)\ar[r]\ar[d,hook]
				& \prod_{p\in P} \LC(\X_p)\ar[d,hook]\\
				\X \ar[r] &
				\prod_{p\in P} \X_p
			\end{tikzcd}
		\end{center}
		is cartesian in $\Cat_\infty$,   
		and by Theorem \ref{cart for LC_B}, the square
		\begin{center}
			\begin{tikzcd}
				\LC_{\Fun(P,\An)} (\X)\ar[r]\ar[d,hook]
				& \prod_{p\in P} \LC(\X_p)\ar[d,hook]\\
				\X \ar[r] &
				\prod_{p\in P} \X_p
			\end{tikzcd}
		\end{center}
		is cartesian in $\Cat_\infty$.  
		Thus the canonical functor
		$\LC_{\Fun(P,\An)} (\X)\to \Cons_P(\X)$
		is an equivalence.
	\end{proof}

	\small
	\bibliographystyle{amsalpha}
	%\clearpage
	%\phantomsection
	%\addcontentsline{toc}{section}{References}
	\bibliography{stratified.bib} 
\end{document}